\documentclass[12pt, reqno]{amsart}
\usepackage{amsmath, amsthm, amscd, amsfonts, amssymb, graphicx, color, enumerate}
\usepackage[bookmarksnumbered, colorlinks, linkcolor=blue, anchorcolor=blue, citecolor=blue, plainpages]{hyperref}

\newtheorem{theorem}{Theorem}[section]
\newtheorem{lemma}[theorem]{Lemma}
\newtheorem{proposition}[theorem]{Proposition}
\newtheorem{corollary}[theorem]{Corollary}
\theoremstyle{definition}
\newtheorem{definition}[theorem]{Definition}

\theoremstyle{remark}

\numberwithin{equation}{section}

\def\f#1#2{\frac{#1}{#2}}

\def\mc#1{\mathcal{#1}}

\def\td{\tilde}

\def\a{\alpha}

\def\p#1{\partial #1}

\def\De{\Delta}

\def\G{\Gamma}
\def\g{\gamma}

\def\la{\lambda}
\def\La{\Lambda}

\def\Om{\Omega}

\def\R{\Bbb{R}}

\def\lan{\langle}
\def\ran{\rangle}
\def\ra{\rightarrow}

\def\ol{\overline}

\begin{document}
\setcounter{page}{1}
\centerline{}
\centerline{}

\title[]{The uniqueness of minimal maps
into Cartan-Hadamard manifolds via the squared singular values}

\author{Zhiwei Jia}
\address{School of Mathematical Sciences, Fudan University, Shanghai, China}
\email{\texttt{zwjia18@fudan.edu.cn} (Zhiwei Jia)}

\author{Minghao Li}
\address{Shanghai Center for Mathematical Sciences, Fudan University, Shanghai, China}
\email{\texttt{mhli19@fudan.edu.cn} (Minghao Li)}

\author{Ling Yang}
\address{School of Mathematical Sciences, Fudan University, Shanghai, China}
\address{Shanghai Center for Mathematical Sciences, Fudan University, Shanghai, China}
\email{\texttt{yanglingfd@fudan.edu.cn} (Ling Yang)}

\begin{abstract}
In this paper, we give a uniqueness theorem for the Dirichlet problem of minimal maps into general Riemannian manifolds with non-positive sectional curvature,
improving \cite[Theorem 5.2]{Lee.2019}. The proof of this theorem is based on the convexity of several functions in terms of squared singular values
along the geodesic homotopy of two given minimal maps.

%and then give some applications of the main theorem. The key step is to use the convexity of Jacobian field along the geodesic homotopy to control
%the singular values and thus guarantee the stability of the graph during the variation.
\end{abstract}
\maketitle

\section{Introduction}
 %Let $M$ and $N$ be two Riemannian manifolds of dimension $m$ and $n$, respectively. A smooth map $f:M \rightarrow N$ is called a $minimal$ $map$ if its
%graph in the product manifold $M\times N$ is minimal(i.e. has vanishing mean curvature). If we suppose $M = \Omega (\subset \mathbb{R} ^m)$ is
%a bounded convex domain, $N=\mathbb{R}^n$, and $\phi :\partial \Omega \rightarrow \mathbb{R}^n$, the Dirichlet problem for minimal surface system is
%searching for solution $f:\Omega \rightarrow \mathbb{R}^n$ to satisfy the following equations:
Let $f=(f^1,\cdots,f^n)$ be a vector-valued function on a domain $\Om\subset \R^m$, %into $\R^n$,
$\G(f):=\{(x,f(x)):x\in \Om\}$
be the {\it graph} of $f$, then $\G_f$ is a {\it minimal submanifold} (i.e. the mean curvature vectors vanish everywhere) of $\R^{m+n}$ if and only if $f$
satisfies the minimal surface system as follows:
\begin{equation}
    \sum\limits _{i,j=1}^{m} \frac{\partial}{\partial{x^{i}}} \left( \sqrt{g}g^{ij} \frac{\partial{f^{\alpha}}}{\partial{x^{j}}}\right)=0, \qquad
    \forall \alpha = 1,\cdots, n, \label{eq.1}\\
    %f|_{\partial \Omega} & = & \phi \label{eq.2}
\end{equation}
where $g_{ij}:=\delta_{ij}+\sum\limits_{\alpha=1}^{n}\frac{\partial f^{\alpha}}{\partial x^i}\frac{\partial f^{\alpha}}{\partial x^j}$,
$(g^{ij})$ is the inverse of $(g_{ij})$ and $g := \det(g_{ij})$. $\De_f:=\sqrt{g}$ is called the {\it slope} of $f$.
As a generalization, Schoen \cite{Schoen. 1993} introduced the conception of {\it minimal map}: Let $(M,g_M)$ and $(N,g_N)$ be two Riemannian manifolds of dimension $m$ and $n$, respectively,
then a smooth map $f:M\ra N$ is a minimal map if and only if its graph is a minimal submanifold in the product manifold $M\times N$.
The research on minimal graphs has a long and fertile history and a lot of works focus on the {\it Dirichlet problem}:
{\it Given a bounded domain $\Om$ in $M$ and a map $\phi$ from the boundary of $\Om$ into $N$, what kind of and how many minimal maps exist,
so that each one $f:\Om\ra N$ satisfies $f|_{\p \Om}=\phi$.}

For the case of minimal graphs of codimension 1, i.e. $N=\R$, {\it the solution to the Dirichlet problem is unique (if exists),
and the graph of this solution is area-minimizing}, i.e., the volume of this submanifold takes the minimum along all
submanifolds with the graph of $\phi$ as their boundary. These conclusions can be derived from the convexity of the volume functional
(see e.g. \cite[Chap. 12]{Giusti-Williams. 1984}). Unfortunately, these beautiful results cannot generalized to higher codimensional case,
i.e. $n \geq 2$. In their pioneering paper \cite{LO.1977}, Lawson-Osserman constructed distinct minimal maps $f_1,f_2,f_3$ from the unit disk $\Bbb{D}$
in $\R^2$ to $\R^2$ sharing the same boundary data, where $f_3$ corresponds to an unstable minimal graph, i.e., even a little deformation
can decrease its area. Afterwards, by making systematic developments on Lawson-Osserman
constructions, it was shown in \cite{XYZ. 2019} that there exist uncountable many vector-valued functions on
Euclidean spheres, so that we can find infinitely many solutions to the Dirichlet problem with
each of them as the boundary data. Therefore, it is usually referred to a restricted class of submanifolds when we talk about the uniqueness
and stability of minimal graphs
for the higher codimensional case.

From the viewpoint of singular values, Lee, Wang, Tsui, Ooi and other mathematicians \cite{Lee.2003, Lee.2008, Lee.2014, Lee.2019, Li.2023} studied this problem, making a series of successive works.
Let $df$ be the tangent map from $T_x M$ into $T_{f(x)}N$, then for any non-negative number $\la$,
$\la$ is a {\it singular value} of $df$ at $x$, if and only if $\la^2$ is a critical value of the squared norm function
\begin{equation}
v\mapsto g_N(df(v),df(v))\qquad (\forall v\in T_x M, |v|=1).
\end{equation}
By the theory of singular value decomposition, the slope of $f$ can be written as
\begin{equation}
\De_f=\prod_{i=1}^m (1+\la_i^2)^{\f{1}{2}},
\end{equation}
where $\la_1\geq \cdots\geq \la_m$ are all singular values of $df$ at $x$. % and $dv_M$ is the volume form associated to $g_M$.
As shown in \cite{Lee.2014},
\begin{equation}
(x_1,\cdots,x_m)\mapsto \prod_{i=1}^m (1+x_i^2)^{\f{1}{2}}
\end{equation}
is a strictly convex (or convex) function on $\mc{M}$ (or $\ol{\mc{M}}$), where $\mc{M}$ consists of all vectors in $\R_{\geq 0}^m:=[0,+\infty)^m$ satisfying
%Different from the case of codimension one, Lawson and Osserman showed that the higher codimensional Dirichlet problem for minimal surface system may not
%necessarilly have existence, uniqueness and stability in \cite{LO.1977}, which means it is impossible to have a uniqueness theorem for the entire class
%of minimal surface system. Therefore it is usually referred to a restricted calss of surfaces when we talk about the uniqueness of minimal graph
%for higher codimension.
%Define $\mathcal{M}$ as a subset of $\mathbb{R} _{\ge 0}^m$ satisfying
\begin{eqnarray}
    & x _i x _j < 1,\qquad \forall i \ne j, & \\
    & \prod\limits _{i=1}^{m} (1-x_i^2) + \sum\limits _{i=1}^{m} (1-x_1^2) \cdots x_i^2 \cdots (1-x_m^2) > 0, &
\end{eqnarray}
and $\ol{\mc{M}}$ is the closure of $\mc{M}$. Moreover, in conjunction of the second variation formula of the volume functional in terms of singular values
(see \cite{Lee.2008}), criteria for the stability and uniqueness of minimal graphs can be established as follows.
\begin{theorem}\cite{Lee.2014}
Let $f:M\ra N$ be a minimal map, where $N$ has non-positive sectional curvature everywhere, %i.e., $K_N\leq 0$,
 then the graph of $f$ is stable (or weakly stable)
whenever the singular value vector $(\la_1,\cdots,\la_m)\in \mc{M}$ (or $\ol{\mc{M}}$) everywhere on $M$.
\end{theorem}
%Afterwards, by applying majorization techniques in convex optimisation, a uniqueness theorem for
%minimal maps into Euclidean spaces can be derived:
\begin{theorem}\label{thm-LYZ}\cite{Li.2023}
Suppose $f_0,f_1:\Om (\subset M) \ra \R^n$ are both minimal maps with the same boundary data,
and the singular value vectors of both $f_0$ and $f_1$ all lie in a fixed symmetric convex subset of $\ol{\mc{M}}$, then $f_0=f_1$.
\end{theorem}
To prove Theorem \ref{thm-LYZ}, it is natural to consider the {\it geodesic homopoty} $\{f_t:t\in [0,1]\}$ of $f_0$ and $f_1$,
such that for each $x\in M$, $\g_x(t):=f_t(x)$ is a geodesic connecting $f_0(x)$ and $f_1(x)$.
Along each $\g_x$, denote by $\la(t):=(\la_1(t),\cdots,\la_m(t))$ the singular value vector function,
then it is easy to verify that $\sum\limits_{i=1}^k \la_i(t)$ is a convex function for each $1\leq k\leq m$.
Based on this fundamental fact, we can show the singular value vectors of $f_t$ still lie in $\ol{\mc{M}}$,
and then $f_0=f_1$ follows from the second variation formula. However, if we replace the target manifold of minimal maps
by a general Riemannian manifold $N$ with non-positive sectional curvature, $\sum\limits_{i=1}^k \la_i(t)$
is no longer convex, so the above scheme is not feasible.

To overcome this obstruction, we consider the squared singular value vector
\begin{equation}
\la^2(df):=(\la_1^2,\cdots,\la_m^2)
\end{equation}
in the present paper.
Let $\ol{\mathcal{N}}$ be a subset of $\mathbb{R} _{\ge 0}^m$, which consists of all such vectors $a:=(a_1,\cdots,a_m)$
satisfying the following two conditions:
\begin{eqnarray}
    & a_ia_j \leq 1,\qquad \forall i \ne j, & \label{eq.29}\\
    & \prod\limits _{i=1}^{m} (1-a_i) + \sum\limits _{i=1}^{m} (1-a_1) \cdots a_i \cdots (1-a_m) \geq 0. & \label{eq.30}
\end{eqnarray}
%{\color{red} Denote $\mathcal{N}$ as the subset of $\mathbb{R} _{\ge 0}^m$ where both inequalities hold strictly, and $\partial\mathcal{N}=\ol{\mathcal{N}}-\mathcal{N}$.}
Then obviously
\begin{equation}
\la^2(df)\in \ol{\mathcal{N}}\Longleftrightarrow \la(df)\in \ol{\mc{M}}.
\end{equation}
The main goal of this paper is the proof of the following criterion for the uniqueness of minimal maps into Riemannian manifolds
with nonpositive sectional curvature, which gives some corollaries.
\begin{theorem}\label{main-thm}
    Let $N$ be a complete
    Riemannian manifold with non-positive sectional curvature,
    and $f_0,f_1:\Omega (\subset M) \rightarrow N$ be minimal maps with the same boundary data.
    If $f_0$ is homotopic to $f_1$, and both $\lambda ^2(df_0)$ and $\lambda ^2(df_1)$ lie in a symmetric convex set $\mc C\subset \overline{\mathcal{N}}$, then $f_0=f_1$.
\end{theorem}

\begin{corollary}\label{cor}
 Let $N$ be a complete
    Riemannian manifold with non-positive sectional curvature.
    Suppose $f_0,f_1:\Omega (\subset M) \rightarrow N$ are minimal maps with the same boundary data,
    which are homotopic to each other, then $f_0=f_1$ if either of the following occurs:
    \begin{itemize}
    \item the singular values of $f_0$ and $f_1$ all satisfy $\la_i^2+\la_j^2\leq 2$ ($\forall i\neq j$) and $\sum\limits_{i=1}^m \la_i^2\leq 3-\f{1}{m-1}$;
    \item $m\geq 3$, and the singular values of $f_0$ and $f_1$ all satisfy $\la_i^2+\la_j^2\leq 2$ ($\forall i\neq j$) and $\prod\limits_{i=1}^m (1+\la_i^2)^{\f{1}{2}} \leq \sqrt{3}\left(2-\f{1}{m-1}\right)^{\f{1}{2}}$;
    \item the slopes of $f_1$ and $f_2$ are no more than $\sqrt{3}$.

    \end{itemize}

\end{corollary}
In particular, the first conclusion of Corollary \ref{cor} is an improvement of \cite[Theorem 5.2]{Lee.2019}, which claims $f_0=f_1$
whenever $\sum\limits_{i=1}^m \la_i^2< 2$.

This paper will be organized as follows. In Section \ref{S2}, we prove the existence of geodesic homotopy $\{f_t:t\in [0,1]\}$
of $f_0$ and $f_1$ with the aid of the classical Cartan-Hadamard theorem
and calculate the second derivative of the volume function of $\G(f_t)$ in terms of the singular values.
%and calculate the second variation
%of the volumes of the corresponding graphs along this homotopy in terms of the singular values.
Afterwards in Section \ref{S3}, by applying majorization techniques in convex optimisation as in \cite{Lee.2019,Li.2023},
we establish the following confined property of the squared singular value vector function along the geodesic homotopy:
{\it For any interval $[t_1,t_2]\subset [0,1]$, $\la^2(t_1),\la^2(t_2)\in \ol{\mc{N}}$ implies
$\la^2(t)\in \ol{\mc{N}}$}. In the process, the convexity of the functions $\sum\limits_{i=1}^k \la_i^2(t)$
with $1\leq k\leq m$ plays
a crucial role. %Later in \ref{S4}, in conjunction of the second variation formula, we use the induction method to prove Theorem \ref{main-thm}.
Section \ref{S4} will be dedicated to the proof of Theorem \ref{main-thm} based on preliminary works in the last two sections. %by utilizing the induction method.
Finally in Section \ref{S5}, the construction of symmetric convex subsets of $\ol{\mc{N}}$ enables us to give
applications of Theorem \ref{main-thm}.

%To prove this theorem, we first derive a formula for $\lambda ^2(t)$ in terms of weak majorization and imply that if $\lambda ^2(0)$ and $\lambda ^2(1)$ lie
%in a symmetric convex subset of $\mathcal{N}$(or $\overline{\mathcal{N}}$), $\lambda ^2(t)$ lies in $\mathcal{N}$(or $\overline{\mathcal{N}}$). Then for the case
%that $\lambda ^2(df_0)$ and $\lambda ^2(df_1)$ lie in a symmetric convex subset of $\mathcal{N}$, we can directly conclude the uniqueness by
%the second variation foumula for the volume functional as is done in \cite{Lee.2019}. For the other case, we explore the property of $\lambda ^2(t)$ and show that
%$\lambda ^2(t) \in \partial \mathcal{N}$ implies the vanishing of the variation field, which conclude the uniqueness by induction as is done in \cite{Li.2023}.
%After that, we give some applications of the main theorem. Particularly Corollary \ref{coro.2} covers the condition of \cite[Theorem 5.2]{Lee.2019}.

%the main theorem of this paper%

\centerline{}
\centerline{}
\section{The second variation formula for the volume functional}\label{S2}
%\noindent In this section, we will derive the second variation formula for the volume functional in terms of singular values of the graph function.

Let $\Om$ be a bounded domain of an $m$-dimensional Riemannian manifold
$(M,g_M)$, $(N,g_N)$ be an $n$-dimensional complete Riemanniam manifold with non-positive sectional curvature, i.e. $K_N\leq 0$,
and $f_0,f_1$ be both smooth maps from $\Om$ into $N$.
Assume $f_0,f_1$ are homotopic to each other, and $f_0|_{\p\Om}=f_1|_{\p\Om}$. Let $\td{N}$ be the universal covering manifold
of $N$ equipped with the pull-back metric $g_{\td{N}}$, $\td{f}_0,\td{f}_1:\Om\ra \td{N}$ be lifts of $f_0,f_1$, respectively, so that
$\td{f}_0(x_0)=\td{f}_1(x_0)$ for a fixed $x_0\in \Om$,
 then $K_{\td{N}}\leq 0$ and the classical Cartan-Hadamard theorem implies
the existence and uniqueness of the geodesic $\td{\g}_x(t):[0,1]\ra \td{N}$ for each $x\in \Om$, so that
$\td{\g}_x(0)=\td{f}_0(x)$, $\td{\g}_x(1)=\td{f}_1(x)$.
Define
$$\td{f}_t(x):=\td{\g}_x(t),\quad f_t:=\pi\circ \td{f}_t$$
with $\pi$ the universal covering mapping from $\td{N}$ onto $N$.
Such $\{f_t:M\ra N|t\in [0,1]\}$ is called a {\it geodesic homotopy}, which satisfies
\begin{itemize}
\item $f_t$ smoothly dependents on $t$;
\item for each $y\in \p \Om$, $t\in [0,1]\mapsto f_t(y)$ is a constant function;
\item for each $x\in \Om$, $\g_x(t):=f_t(x)$ is a geodesic in $N$ connecting $f_0(x)$ and $f_1(x)$.
\end{itemize}

%Suppose that $f_0,f_1:\Omega(\subset M)\rightarrow N$ are smooth maps between two Riemannian manifolds $(M,g_M)$ and $(N,g_N)$ of dimension $m$ and $n$
%respectively, where $f_0$, $f_1$ are homotopic and $K_N\le 0$. Let $f_t$ be the homotopy of $f_0$ and $f_1$. On each contractible neighborhood
%$U$($\subset \Omega$) of $x$, we can lift the homotopy $f_t$ to the universal
%covering of $N$ and denote the lifting by $\tilde{f}_t$. By Cartan-Hadamard theorem, the universal covering of $N$ is diffeomorphic to $\mathbb{R}^n$ and thus
%implies a unique geodesic connecting the liftings $\tilde{f}_0(y)$ and $\tilde{f}_1(y)$ for each $y\in U$. It can be verified that if there exists two different
%liftings of $f_t$ on $U$, the images of these liftings in $N$ are still identical. This means we can find a unique homotopy of $f_0$ and $f_1$ (also denoted by
%$f_t$ for brevity) such that $\gamma _x(t):=f_t(x)$ is a geodesic in $N$ connecting $f_0(x)$ and $f_1(x)$ for each $x \in \Omega$ by the given homotopy.
%Such $f_t$ is called $a$ $geodesic$ $homotopy$.

Given $0 \le t \le 1$, then $f_t$ induces an embedding %$\hat{f}_t: \Omega \rightarrow M \times N$
%such that $\hat{f}_t(x):=(x,f_t(x))$.
$$x\in \Om \mapsto (x,f_t(x))\in M\times N,$$
whose image, denoted by $\Gamma _{f_t}$, is the graph of $f_t$.
%The images of $\Omega$ in $M \times N$ is called $the$ $graph$ of $f_t$, denoted by $\Gamma _{f_t}$.
For each $x \in \Omega$, by the theory of singular value decomposition, there exist orthonormal bases $\left\{a_i\right\}_{i=1}^{m}$,
$\left\{b_j\right\}_{j=1}^{n}$ in $T_xM$ and $T_{f_t(x)}N$ respectively, such that
%and $m$ singular values
%$\lambda_1\left(t\right)\ge\cdots\ge\lambda_r\left(t\right)>\lambda_{r+1}\left(t\right)=\cdots=\lambda_m\left(t\right)=0$ where $r=rank\,df_t$ such that
\begin{equation}\label{eq7}
    df_t\left(a_i\right)=\left\{
        \begin{matrix}
            &\lambda_i(t)b_i, \hfill & \hspace{3em}   &i=1,\cdots, r,\hfill\\
            &0, \hfill & \hspace{3em}   &i=r+1,\cdots, m.\hfill
        \end{matrix}\right.
\end{equation}
where
\begin{equation}
\lambda_1\left(t\right)\ge\cdots\ge\lambda_r\left(t\right)>\lambda_{r+1}\left(t\right)=\cdots=\lambda_m\left(t\right)=0
\end{equation}
are singular values of $(df_t)_x:(T_x,g_M)\ra (T_{f_t(x)},g_N)$, and $r$ is the rank of this tangent map.
As in \cite{Lee.2019},
\begin{equation}
\lambda\left(t\right):=(\lambda_1(t),\cdots,\lambda_m(t))
\end{equation}
is called $the$ $singular$ $value$ $vector$ of $df_t$ at $x$.
%The metric on $M\times N$ naturally induces the metric
%$g(t)=g_M+f_t^*g_N$ on the graph. We have
Let
$$g(t):=g_M+f_t^*g_N$$
be the induced metric on $\G_{f_t}$, whose corresponding volume form is
\begin{equation}\label{eq6}
    dv_t=\sqrt{\det(g_{ij}(t))}\,dv_M
\end{equation}
with $dv_M$ the volume form of $M$ and
\begin{equation}
    g_{ij}\left( t\right):=g_M\left( a_i,a_j\right)+g_N\left( df_t\left( a_i\right),df_t\left( a_j\right) \right).
\end{equation}
%where $dv_t$ and $dv_M$ are the volume forms of $\Gamma_{f_t}$ and $M$ respectively.
Denote by $V:=\frac{df_t}{dt}$ the variation field on $\Om$, then a straightforward calculation shows
\begin{eqnarray}
    \frac{d}{dt}g_{ij}\left( t\right) & = & \left\langle \nabla _{df_t(a_i)}V,df_t(a_j)\right\rangle +
    \left\langle df_t(a_i),\nabla _{df_t(a_j)}V\right\rangle,\label{eq11} \\
    \frac{d^2}{dt^2}g_{ij}(t)         & = & 2\left\langle \nabla _{df_t(a_i)}V ,\nabla _{df_t(a_j)}V\right\rangle\nonumber\\
    && + \left\langle R(V,df_t(a_i))V,df_t(a_j)\right\rangle + \left\langle df_t(a_i),R(V,df_t(a_j))V\right\rangle\label{eq12}\\
    && + \left\langle \nabla _{df_t(a_i)} \nabla _V V, df_t(a_j)\right\rangle + \left\langle df_t(a_j), \nabla _{df_t(a_j)} \nabla _V V
    \right\rangle. \nonumber
\end{eqnarray}
Here $\lan\cdot,\cdot\ran:=g_N(\cdot,\cdot)$,
 $\nabla$ is the Levi-Civita connection associated to $g_N$ and
\begin{equation}
    R(X,Y) := \nabla _{X} \nabla _{Y} - \nabla _{Y} \nabla _{X} - \nabla _{\left[ X,Y \right]}
\end{equation}
is the corresponding curvature operator.
%Calculating the first and second derivatives of $g_{ij}\left( t\right)$,
%The volume functional of the graph is
Denote by
\begin{equation}
    A(t):=\int_\Omega \sqrt{\det(g_{ij}(t))}\,dv_M.
\end{equation}
the volume of $\G_{f_t}$, then
%Therefore, we have the first and second variation formulas as follows:
\begin{equation}
    \frac{d}{dt}A\left( t\right)     =  \frac{1}{2} \int_{\Omega} \sum_{i,j=1}^{m} \left(g^{ij}\frac{dg^{ij}}{dt}\right)dv_t,\label{eq8}
\end{equation}
\begin{equation}\label{eq9}\aligned
    \frac{d^2}{dt^2}A\left( t\right)  = & -\frac{1}{2} \int_{\Omega} \sum_{i,j,k,l=1}^{m} g^{ij}\frac{dg_{jk}}{dt} g^{kl}\frac{dg_{li}}{dt} dv_t
    +\frac{1}{4} \int_{\Omega}\left(\sum_{i,j=1}^{m} g^{ij}\frac{dg_{ij}}{dt}\right) ^2dv_t \\&+\frac{1}{2}\int_{\Omega} \sum_{i,j=1}^{m} g^{ij} \frac{d^2g_{ij}}{dt^2} dv_t\endaligned
\end{equation}
with $\left( g^{ij}\right)$ the inverse of $\left( g_{ij}\right)$.
Denoting
\begin{equation}
p_{i\alpha}:=\left\langle \nabla_{df_t\left( a_i\right)} V,b_{\alpha}\right\rangle
\end{equation}
and replacing $\frac{dg_{ij}}{dt}$ and $\frac{d^2g_{ij}}{dt^2}$ in \eqref{eq9} with formulas in \eqref{eq11}, \eqref{eq12}, we have
\begin{equation}\label{eq.18}
    \frac{d^2}{dt^2}A(t)=(i)+(ii)+(iii)+(iv)+(v),
\end{equation}
where
\begin{eqnarray}
    (i) & = & \int_{\Omega } \left( \sum_{1\leq i \le r}\frac{p_{ii}^2}{(1+\lambda_i^2)^2}+\sum_{1 \le i,j \le r,i \ne j}
    \frac{\lambda_i \lambda_j p_{ii}p_{jj}}{(1+\lambda_i^2)(1+\lambda_j^2)} \right)dv_t, \label{eq.32} \\
    (ii) & = & \int _{\Omega} \sum _{1 \le i < j \le r} \frac{p_{ij}^2+p_{ji}^2-2\lambda _i \lambda _j p_{ij}p_{ji}}{(1+\lambda _i^2)(1+\lambda _j^2)}dv_t,
    \label{eq.33} \\
    (iii) & = & \int _{\Omega} \sum _{1 \le i \le r, r+1 \le \alpha \le n} \frac{p_{i\alpha}^2}{1+\lambda _i^2} dv_t, \label{eq.34} \\
    (iv) & = & \int _{\Omega} \sum _{1 \le i \le r} \frac{\lambda _i^2}{1+\lambda _i^2} \left\langle \nabla _{b_i} \nabla _{V} V, b_i\right\rangle dv_t,
    \label{eq.35} \\
    (v) & = & \int _{\Omega} \sum _{1 \le i \le r} \frac{\lambda _i^2}{1+\lambda _i^2} \left\langle R(b_i,V)b_i,V\right\rangle dv_t. \label{eq.36}
\end{eqnarray}
%$(i)$ is a quadratic form and $(i) \ge 0$ for any $v=(p_{11},\cdots ,p_{rr})$ if $\lambda \in \overline{\mathcal{M}}$ as is shown in \cite{Lee.2014},
%$(ii) \ge 0$ if $0 \le \lambda _i \lambda _j \le 1$, $(iii)$ is non-negative, $(iv)=0$ since $f_t$ is a geodesic homotopy,
%and $(v) \ge 0$ since $K_N \le 0$.
As shown in \cite{Lee.2014}, $(i)\geq 0$ whenever $\lambda \in \overline{\mathcal{M}}$,
$(ii) \ge 0$ whenever $0 \le \lambda _i \lambda _j \le 1$ for each $i\neq j$, $(iii)$ is non-negative, $(iv)=0$ since $\{f_t\}$ is a geodesic homotopy,
and $(v) \ge 0$ since $K_N \le 0$.

\centerline{}
\centerline{}
\section{Confined properties of squared singular value vectors}\label{S3}
%In this section, we explore the properties of $\lambda ^2(df_t)$ as $t$ ranges from 0 to 1.

\begin{definition}\label{max_F}
    Let $x:=(x_1,\cdots ,x_m),\,y:=(y_1,\cdots ,y_m)\in \mathbb{R}^m$, $y$ is called $l$-$weakly$ $majorized$ by $x$, whenever
    $$\sum _{i=1}^k \tilde{y}_i \le \sum _{j=1}^k \tilde{x} _j$$
    for $k=1,2,\cdots ,l$, where $\left\{ \tilde{x} _i\right\}$ (or $\left\{ \tilde{y} _i\right\}$) is a rearrangement of $\left\{ x_i\right\}$
    (or $\left\{ y_i\right\}$) in descending order, denoted by $y \prec _l x$.
    %If both $x \prec _l y$ and $y \prec _l x$ hold, we denote this relation by $x \asymp _l y$.
\end{definition}
For an arbitrary point $x \in \mathbb{R} _{\ge 0}^m$, let $W(x) := \left\{ y \in \mathbb{R} _{\ge 0}^m : y \prec _m x\right\}$, $E(x)$ be the set
consisting of all these points $(\delta _1 x_{\sigma (1)}, \cdots, \delta _m x_{\sigma (m)})$, where $\sigma$ is an arbitrary permutation of $\{1, \cdots, m\}$ and
$\delta _i = 0$ or $1$, and $H(x)$ be the convex hull of $E(x)$, then $W(x)=H(x)$ (see \cite[Theorem 6]{Mirsky.1959}). The following two
lemmas on $W(x)$ shall play a crucial part on the present paper.
%We have $W(x)$ equals the convex hull of $E(x)$ by. A lemma follows:

\begin{lemma}
Let $D$ be a domain of $\R^m$, $F$ be a strictly convex, symmetric function on $D$, such that for each $u,v\in D$, $u_i\leq v_i$ for each $i=1,\cdots,m$ implies
$F(u)\leq F(v)$, and the equality holds if and only if $u=v$. Then for each $x\in D$ satisfying $W(x)\subset D$, we have
$F(y)\leq F(x)$ for each $y\in W(x)$, where the equality holds if and only if $y$ is a rearrangement of $x$.
\end{lemma}

\begin{proof}
Denote
   \begin{equation}\label{v}
   E(x)=\left\{ v_1, \cdots, v_p\right\}
   \end{equation}
   with
   \begin{equation}
   v_\a = (\delta _{1}^{\a} x_{\sigma _\a(1)}, \cdots, \delta _{m}^{\a} x_{\sigma _\a(m)}),
   \end{equation}
   then each $y\in W(x)=H(x)$ can be written as
   \begin{equation}
   y=\la_1 v_1+\cdots+\la_p v_p,
   \end{equation}
where $\la_1,\cdots,\la_p$ are all non-negative numbers, satisfying
$\sum\limits_{\a=1}^p \la_\a=1.$  Let
\begin{equation}
w_\a:= ( x_{\sigma _\a(1)}, \cdots, x_{\sigma _\a(m)})
\end{equation}
and
\begin{equation}\label{z}
z:=\la_1 w_1+\cdots+\la_p w_p\in H(x).
\end{equation}
Combining (\ref{v})-(\ref{z}) we have $y_i\leq z_i$ for all $i=1,\cdots,m$.
Since $F$ is strictly convex and symmetric, we have
$$F(y)\leq F(z)\leq F(x),$$
and the equality holds if and only if $y=z=w_\a$ for some $\a$,
i.e. $y$ is a rearrangement of $x$.
\end{proof}

\begin{lemma}\label{le.1}
    %Suppose $x$ lies in a symmetric convex subset $\mathcal{C}$ of $\overline{\mathcal{N}}$, then $W(x)\subset \overline{\mathcal{N}}$.
Let $\mc C$ be a symmetric convex subset of $\ol{\mc{N}}$ defined in \eqref{eq.29}-\eqref{eq.30} and $x\in \mc C$, then $W(x)\subset \overline{\mathcal{N}}$.
Moreover, if $y\in W(x)\cap \p\mc{N}$ and $\max\{y_i\}>1$, then $\sum\limits_{i=1}^m y_i=\sum\limits_{i=1}^m x_i$.
Here $\p{\mc{N}}$ consists of all such vectors $a:=(a_1,\cdots,a_m)$ in $\ol{\mc{N}}$ satisfying
\begin{equation}
\max_{1\leq i<j\leq m}a_ia_j=1
\end{equation} 
or
\begin{equation}
\prod\limits _{i=1}^{m} (1-a_i) + \sum\limits _{i=1}^{m} (1-a_1) \cdots a_i \cdots (1-a_m)=0.
\end{equation}

\end{lemma}
\begin{proof}
 Let $y$ be an arbitrary point in $W(x)$. Since $W(x)$ is preserved under the action of
 permutations, without loss of generality we can assume
 \begin{equation}
 x_1\geq\cdots \geq x_m, \quad y_1\geq \cdots\geq y_m.
 \end{equation}
If $y_1\leq 1$, then $y$ automatically satisfies \eqref{eq.29}-\eqref{eq.30}
and hence $y\in \ol{\mc{N}}$. Now we assume $y_1>1$.
In this case, noting that
\begin{equation}\aligned
&\prod_{i=1}^m(1-y_i)+\sum_{i=1}^m(1-y_1)\cdots y_i\cdots (1-y_m)\\
=&\prod_{i=1}^m(1-y_i)\left(\sum_{i=1}^m \f{1}{1-y_i}-m+1\right).
\endaligned
\end{equation}
we have
\begin{equation}\label{pN}\aligned
y\in \ol{\mc{N}}\Longleftrightarrow &y_1y_2\leq 1\text{ and }\f{1}{1-y_1}+G(y_2,\cdots,y_m)\leq m-1,\\
y\in \p{\mc{N}}\Longleftrightarrow &y_1y_2=1\text{ or }\f{1}{1-y_1}+G(y_2,\cdots,y_m)= m-1,
\endaligned
\end{equation}
with
\begin{equation}\label{def-G}
G:(y_2,\cdots,y_m)\in [0,1)^{m-1}\mapsto \sum_{i=2}^m \f{1}{1-y_i}.
\end{equation}
Due to the symmetric and convexity of $\mc{C}$,
\begin{equation}\label{hat-x}
\hat{x}:=(x_2,x_1,x_3,\cdots,x_m)
\end{equation}
and
\begin{equation}
\f{x+\hat{x}}{2}=\left(\f{x_1+x_2}{2},\f{x_1+x_2}{2},x_3,\cdots,x_m\right)
\end{equation}
both lie in $\mc{C}$, then
\begin{equation}\label{y1y2}
y_1y_2<\left(\f{y_1+y_2}{2}\right)^2\leq  \left(\f{x_1+x_2}{2}\right)^2\leq 1.
\end{equation}
Denote
\begin{equation}
\aligned
z=&(z_1,z_2,z_3,\cdots,z_m)\\
:=&(y_1,x_2+(x_1-y_1),x_3,\cdots,x_m),
\endaligned
\end{equation}
then $z$ is also a convex combination of $x$ and $\hat{x}$ and hence $z\in \mc{C}$.
By $y\prec_m x$, it is easy to verify that $y\prec_m z$ and
\begin{equation}
(y_2,\cdots,y_m)\prec_{m-1}(z_2,\cdots,z_m).
\end{equation}
Applying Lemma \ref{max_F} to the function $G$ gives
\begin{equation}\label{G}
\f{1}{1-y_1}+G(y_2,\cdots,y_m)\leq \f{1}{1-z_1}+G(z_2,\cdots,z_m)\leq m-1
\end{equation}
and then $y\in \ol{\mc{N}}$. Particularly if $y\in \p\mc{N}$,
then \eqref{pN} and \eqref{y1y2} implies $\f{1}{1-y_1}+G(y_2,\cdots,y_m)=m-1$.
Thus the equality of \eqref{G} holds, which forces $y=z$ and then
\begin{equation}
\sum\limits_{i=1}^m y_i=\sum\limits_{i=1}^m z_i=\sum_{i=1}^m x_i.
\end{equation}

\end{proof}

Now we consider the squared singular value vectors
\begin{equation}
\la^2(t):=(\lambda _1^2(t), \cdots, \lambda _m^2(t))
\end{equation}
along a given geodesic homotopy. Based on Lemma \ref{le.1}, we can get a confined property for
$\la^2(t)$ as follows.
%Now we give some conventions. Denote the singular value vector of $f_t$ by
%$\lambda (t) = (\lambda _1(t), \cdots, \lambda _m(t))$ such that
%$\lambda _1(t) \ge \cdots \ge \lambda _m(t)$ for $t \in \left[ 0,1\right]$. Let $\lambda ^2(t) = (\lambda _1^2(t), \cdots, \lambda _m^2(t))$ and
%$\mu (t) = \frac{t_2-t}{t_2-t_1}\lambda ^2(t_1) + \frac{t-t_1}{t_2-t_1}\lambda ^2(t_2)$, where $0 \le t_1 < t_2 \le 1$ and
%$\mu _i(t)$ is the $i^{th}$ component of $\mu (t)$. Now the proposition
%is given as follows:
\begin{proposition}\label{prop.1}
Let $\left[ t_1,t_2\right] \subset \left[ 0,1\right]$ and
$$\mu(t):=\frac{t_2-t}{t_2-t_1}\lambda ^2(t_1) + \frac{t-t_1}{t_2-t_1}\lambda ^2(t_2)$$
be the linear function on this interval satisfying $\mu(t_1)=\la^2(t_1)$ and $\mu(t_2)=\la^2(t_2)$,
then
$$\sum\limits _{i=1}^{l} \lambda _i^2(t)\leq \sum\limits _{i=1}^{l} \mu _i(t)\qquad \forall 1\leq l\leq m.$$
Especially if $\sum\limits _{i=1}^{l} \lambda _i^2(t_0) = \sum\limits _{i=1}^{l} \mu _i(t_0)$ for some $t_0 \in ( t_1, t_2)$,
    we have $\nabla _{df_t(a_i)}V=0$ for $i=1, \cdots , l$, $t \in \left[ t_1,t_2\right]$.

 Moreover,  if both $\lambda ^2(t_1)$ and $\lambda ^2(t_2)$ lie in a symmetric convex subset $\mc{C}$ in
    $\overline{\mathcal{N}}$, then $\lambda ^2(t) \in \overline{\mathcal{N}}$ for each $t\in [t_1,t_2]$.
  Especially if $\lambda ^2(t_0)\in \partial \mathcal{N}$ and $\lambda _1^2(t_0)>1$ for some $t_0 \in (t_1, t_2)$,
     we have $\nabla _{df_t(a_i)}V=0$ for $i=1, \cdots , m$, $t \in \left[ t_1,t_2\right]$.

    %Let $t\in \left[ t_1,t_2\right] \subset \left[ 0,1\right]$. If both $\lambda ^2(t_1)$ and $\lambda ^2(t_2)$ lie in a symmetric convex subset $\mc{C}$ in
    %$\overline{\mathcal{N}}$, then $\lambda ^2(t) \in \overline{\mathcal{N}}$. Moreover,
    %\begin{enumerate}[(i)]
    %    \item if $\sum\nolimits _{i=1}^{l} \lambda _i^2(t_0) = \sum\nolimits _{i=1}^{l} \mu _i(t_0)$ for some $t_0 \in \left[ t_1, t_2\right]$,
    %    we have $\nabla _{df_t(a_i)}V=0$ for $i=1, \cdots , l$, $t \in \left[ t_1,t_2\right]$;
    %    \item if $\lambda ^2(t_0)\in \partial \mathcal{N}$ and $\lambda _1^2(t_0)>1$ for some $t_0 \in \left[ t_1, t_2\right]$,
    %    we have $\nabla _{df_t(a_i)}V=0$ for $i=1, \cdots , m$, $t \in \left[ t_1,t_2\right]$.
    %\end{enumerate}
\end{proposition}
\begin{proof}
    For any fixed $t_0\in \left[ t_1,t_2\right]$,
    let
    $\left\{ a_i\right\}$ be an orthonormal basis of $T_xM$, such that
    \begin{equation}
        \lan df_{t_0}(a_i),df_{t_0}(a_i)\ran=\left\{
        \begin{matrix}
            &\lambda _i^2(t_0) \hfill & \, & 1\le i \le r:=\text{rank} \, df_{t_0}, \hfill \\
            &0 \hfill & \, & r+1\le i\le m. \hfill
        \end{matrix}
        \right.
    \end{equation}
    %For $t\in \left[ t_1,t_2\right]$, define
    Now we consider two functions
    \begin{eqnarray}
        F_k(t) & := & \sum_{i=1}^{k} \left\langle df_t(a_i),df_t(a_i)\right\rangle, \\
        S_k(t) & := & \sum_{i=1}^{k} \lambda _i^2(t)
    \end{eqnarray}
    on $\left[ t_1,t_2\right]$. Due to the properties of singular values,
    it is easily verified that $S_k(t_0)=F_k(t_0)$ and $F_k(t)\le S_k(t)$. %Moreover $F_k(t)$ is a convex function since
    On the other hand,
    \begin{equation}\label{eq.17}
    \aligned
        \frac{d^2}{dt^2}F_k(t)  = & \frac{d^2}{dt^2}\sum_{i=1}^{k} \left\langle df_t(a_i),df_t(a_i)\right\rangle=  \frac{d}{dt}\sum_{i=1}^{k} 2\left\langle \nabla _{df_t(a_i)}V,df_t(a_i)\right\rangle \\
                                = & 2\sum_{i=1}^{k} \left\langle \nabla _{df_t(a_i)} V,\nabla _{df_t(a_i)} V \right\rangle +
                               2\sum_{i=1}^{k} \left\langle \nabla _V\nabla _{df_t(a_i)}V,df_t(a_i)\right\rangle \\
                                = & 2\sum_{i=1}^{k} \left| \nabla _{df_t(a_i)}V\right| ^2 + 2\sum_{i=1}^{k}\left\langle R(V,df_t(a_i))V,df_t(a_i)\right\rangle
                               \\
                                 & + 2\sum_{i=1}^{k} \left\langle \nabla _{df_t(a_i)}\nabla _VV,df_t(a_i)\right\rangle\geq 0,
    \endaligned
    \end{equation}
    %The third item of the last equation vanishs since $f_t$ is a geodesic homotopy. The second item of the last equation is
    %non-negative for $K_N\le 0$. Hence,
    showing $F_k(t)$ is a convex function. Thus%By the properties above, we have
    \begin{equation}\label{eq.16}
    \aligned
        S_k(t_0) = F_k(t_0)  \le & \frac{t_2-t_0}{t_2-t_1}F_k(t_1) + \frac{t_0-t_1}{t_2-t_1}F_k(t_2) \\
         \le & \frac{t_2-t_0}{t_2-t_1}S_k(t_1) + \frac{t_0-t_1}{t_2-t_1}S_k(t_2) \\
         = & \sum_{i=1}^{k} \mu _i(t_0).
        \endaligned
    \end{equation}
    Moreover, if $S_l(t_0) = \sum\limits _{i=1}^{l} \mu _i(t_0)$, then \eqref{eq.16}
    shows $F_l|_{[t_1,t_2]}$ is linear and hence \eqref{eq.17} implies $\nabla _{df_t(a_i)}V=0$ for $i=1, \cdots , l$.

    %which means $\lambda ^2(t) \prec _m\mu (t)$ for $t\in \left[ t_1,t_2\right]$. Since $\lambda ^2(t_1)$ and $\lambda ^2(t_2)$ lie in a symmetric convex
    %subset of $\overline{\mathcal{N}}$, $\mu(t)$ also lies in the symmetric convex subset and hence $\lambda ^2(t) \in \overline{\mathcal{N}}$
    %by Lemma \ref{le.1}.  This completes the proof of $(i)$.

Note that (\ref{eq.16}) is equivalent to saying that $\lambda ^2(t) \prec _m\mu (t)$ for $t\in \left[ t_1,t_2\right]$.
Once $\la^2(t_1),\la^2(t_2)\in \mc{C}$,
the symmetry and convexity of $\mc{C}\subset \ol{\mc{N}}$ shows $\mu(t)\in \mc{C}$ and hence we get $\lambda ^2(t) \in \overline{\mathcal{N}}$ by Lemma \ref{le.1}.
Moreover, if $\lambda ^2(t_0)\in \partial \mathcal{N}$ and $\lambda _1^2(t_0)>1$, then
$S_m(t_0) =\sum\limits_{i=1}^m \la_i^2(t_0)= \sum\limits _{i=1}^{m} \mu_i(t_0)$ and hence
 $\nabla _{df_t(a_i)}V=0$ for $i=1, \cdots , m$ and $t \in \left[ t_1,t_2\right]$.

\end{proof}

\centerline{}
\centerline{}
\section{Proof of the main theorem}\label{S4}
%In this section, we will prove the main theorem using the proposition above.
%\begin{theorem}\label{theorem.1}
%    Suppose that $f_0,f_1:\Omega (\subset M)\rightarrow N$ are minimal maps with the same boundary data, $f_0$ and $f_1$ are homotopic, $K_N\le 0$.
%    If both $\lambda ^2(df_0)$ and $\lambda ^2(df_1)$ lie in a symmetric convex subset $\mc C$ of $\overline{\mathcal{N}}$, then $f_0=f_1$.
%\end{theorem}
%\begin{proof}
Suppose that $f_0,f_1:\Omega (\subset M)\rightarrow N$ are minimal maps, such that $f_0,f_1$ are homotopic to each other and
$f_0|_{\partial \Om}=f_1|_{\partial \Om}$.
Let $\{f_t:\Om\ra N|t\in [0,1]\}$ be a geodesic homotopy of $f_0$ and $f_1$.
For each $x\in \Om$, let $\la^2(t):=\la^2((df_t)_x)$ be the squared singular value vector function.
By Proposition \ref{prop.1}, $\la^2(0),\la^2(1)\in \mc C$ ensures $\lambda ^2(t) \in \overline{\mathcal{N}}$,
i.e., $\lambda (t) \in \overline{\mathcal{M}}$ % As shown in Section \ref{S2},
and hence $\frac{d^2}{dt^2} A(t) \ge 0$. In conjunction with $\frac{d}{dt}\Big|_{t=0} A(t) =\frac{d}{dt}\Big|_{t=1} A(t)= 0$
(since both $f_0$ and $f_1$ are minimal maps), we have $\frac{d^2}{dt^2} A(t) = 0$, i.e.
$(i) = (ii) = (iii) = 0$ for each $t\in (0,1)$ (see \eqref{eq.32}-\eqref{eq.34}).

    %Recall the second variation formula for the volume functional of $f_t$ in section 2. If $f_0$ and $f_1$ are minimal maps, the first variation formula
    %$\frac{d}{dt} A(t) = 0$ for $t = 0,1$. Moreover if $\lambda ^2(0)$ and $\lambda ^2(1)$ lie in a symmetric convex subset of $\overline{\mathcal{N}}$,
    %$\lambda ^2(t) \in \overline{\mathcal{N}}$ for $0 \le t \le 1$ by Proposition \ref{prop.1} which means that $\lambda (t) \in \overline{\mathcal{M}}$ and
    %hence $\frac{d^2}{dt^2} A(t) \ge 0$. Thus the condition of this theorem implies that $\frac{d^2}{dt^2} A(t) = 0$ for $0 \le t \le 1$.

    %For the squared singular value function,
    Define
    \begin{equation}
    \aligned
        \Lambda _1 & :=  \left\{ t\in (0,1) : \lambda(t)\in \mathcal{M}\right\}, \\
        \Lambda _2 & :=  \left\{ t\in (0,1) : \lambda _2(t) < 1 \right\}, \\
         & \cdots & \\
        \Lambda _m & :=  \left\{ t\in (0,1) : \lambda _m(t) < 1 \right\}, \\
        \Lambda _{m+1} &:= (0,1).
    \endaligned
    \end{equation}

    When $t\in \Lambda _1$, as  shown in \cite[Theorem 3.2]{Lee.2014}, $(i) = (ii) = (iii) = 0$ implies that
    $p_{i\alpha} = 0$ for any $1 \le i \le m$ and $1 \le \alpha \le n$ and hence
    \begin{equation}\label{dV}
    \nabla _{df_{t}(a_i)} V = 0 \qquad \forall 1 \le i \le m.
    \end{equation}
    By the continuity, this equality holds for each $t\in \overline{\Lambda}_1$.

    If $\Lambda _2 \setminus \overline{\Lambda}_1=\emptyset$, \eqref{dV} always holds in $\La_2$.
    Otherwise, for each $t$ in this set, %there exists an closed interval $\left[ t_1,t_2\right]$, such that
     %$t\in \left[ t_1,t_2\right] \subset \Lambda _2 \setminus \overline{\Lambda}_1$.
     %By Proposition \ref{prop.1},
    we have $\lambda ^2(t)\in \partial \mathcal{N}$, then $\la_2^2(t)<1$ forces $\la_1^2(t)>1$,
    %. If $\lambda _1 \le 1$, then $\lambda _i(t) \lambda _j(t) \le \lambda _1(t) \lambda _2(t) < 1$ since
    %$\lambda _1(t) \ge \cdots \ge \lambda _m(t) \ge 0$. We can conclude that $\lambda ^2(t) \in \mathcal{N}$, it conflicts with
    %$\lambda ^2(t)\in \partial \mathcal{N}$. So we have $\lambda _1^2(t) > 1$.
    then Proposition \ref{prop.1} and the continuity ensure \eqref{dV} holds for all $t \in \overline{\Lambda}_2$.

    Next we need to show that the equality \eqref{dV} also holds on each $\Lambda _i$ by induction on $i$.
    Suppose that $\nabla _{df_t(a_i)} V = 0$ holds for all $t\in \overline{\Lambda}_k$ with
    $2 \le k \le m$ and the open set $\Lambda _{k+1} \setminus \overline{\Lambda}_k$ is nonempty. For each $t\in \left[ t_1,t_2\right] \subset \Lambda _{k+1}
    \setminus \overline{\Lambda}_k$, it is easy to see that $\lambda _1(t) = \cdots = \lambda _k(t) = 1$ and
    $1 > \lambda _{k+1}(t) \ge \cdots \ge \lambda _m(t) \ge 0$, which means $\sum\limits _{i=1}^{k} \lambda _i^2(t) = \sum\limits _{i=1}^{k} \mu _i(t)$,
    and we can conclude that
    \begin{equation}
        \nabla _{df_t(a_i)}V = 0, \qquad \forall 1 \le i \le k, \,\, t \in \Lambda _{k+1} \label{eq.37}
    \end{equation}
    by Proposition \ref{prop.1}. %(i.e. $p_{i\alpha} = 0$ for $ 1 \le i \le k$ and $1 \le \alpha \le n$).
    In combination of \eqref{eq.32}, \eqref{eq.33} and \eqref{eq.34}, we have $p_{i \alpha} = 0$ for $k+1 \le i \le r$ and
    $1 \le \alpha \le n$, which means $\nabla _{df_t(a_i)} V = 0$ for $k+1 \le i \le m$. Together with \eqref{eq.37} we know \eqref{dV}
    also holds for $t \in \Lambda_{k+1}$, finishing the induction step.

    %By induction we have $\nabla _{df_t(a_i)} V = 0$ for $1 \le i \le m$ and $0 \le t \le 1$.
    Therefore, for each $t\in \left[ 0, 1 \right]$, $V$ is a parallel vector field on the graph $\Gamma _{f_t}$. According to the boundary condition
    we can derive $V \equiv 0$ and hence $f_0 = f_1$. This completes the proof of %the main theorem, i.e.,
    Theorem \ref{main-thm}.
%\end{proof}

\centerline{}
\section{Applications}\label{S5}
In this section, we give some applications of Theorem \ref{main-thm}.
%the main theorem.

%In \cite{Lee.2019} Lee gave several symmetric convex subset of $\mathcal{M}$ as applications, here follows two of these symmetric convex subsets:
%\begin{enumerate}
%    \item $\left\{ \lambda \in \mathbb{R} _{\ge 0}^m : \sum _{i=1}^{m} \lambda _i^2 < 2\right\} \subset \mathcal{M}$;
%    \item $\left\{ \lambda \in \mathbb{R} _{\ge 0}^m : \prod_{i=1}^{m} (1 + \lambda _i^2)^{\frac{1}{2}} < 2 \right\} \subset \mathcal{M}$.
%\end{enumerate}
%It means that $\left\{ a \in \mathbb{R} _{\ge 0}^m : \sum a _i \le 2 \right\}$ and
%$\left\{ a \in \mathbb{R} _{\ge 0}^m : \prod (1 + a _i)^{\frac{1}{2}} \le 2 \right\}$ are all symmetric subsets of $\overline{\mathcal{N}}$. Easy to see that
%the first subset is convex however the second is not. We now do some improvements to these sets and apply our main theorem to them.

%Notice that if $a$ lies in a symmetric convex subset of $\overline{\mathcal{N}}$, $a_i + a_j \le 2$ for $1 \le i < j \le m$. Denote $T :=
%\left\{ a \in \mathbb{R} _{\ge 0}^m : a_i + a_j \le 2, 1 \le i < j \le m \right\}$. A feasible approach is to consider
%the intersection of $a_i + a_j \le 2$ and $\sum a_i \le 2 + \alpha$, $\alpha > 0$ is a constant. In this way we have the following corollary:

Let $\mc C$ be a symmetric convex subset of $\ol{\mc N}$, then for each $a\in \mc C$, we can proceed as in
\eqref{hat-x}-\eqref{y1y2} to show
\begin{equation}\label{sum}
a_i+a_j\leq 2\qquad \forall 1 \leq i < j \leq m.
\end{equation}
On the other hand,
from this condition, it immediately follows that
   \begin{equation}
   a_ia_j\leq \left(\f{a_i+a_j}{2}\right)^2\leq 1,
   \end{equation}
i.e. such $a$ must satisfy \eqref{eq.29}. It is natural to ask, besides (\ref{sum}), whichever additional restrictions can make sure
a symmetric convex subset of $\R_{\geq 0}^m$ completely lies in $\ol{\mc N}$. In the following text we shall consider this question.

\begin{corollary}\label{coro.1}
    Suppose that $f_0,f_1:\Omega (\subset M)\rightarrow N$ are minimal maps with the same boundary data, $f_0$ is homotopic to $f_1$ and $K_N \le 0$.
    If both $\lambda ^2(df_0)$ and $\lambda ^2(df_1)$ lie in
    $$\mc{C}_m:=\left\{a\in \R_{\geq 0}^m\Big|\sum\limits_{i=1}^m a_i\leq 3-\f{1}{m-1},a_i+a_j\leq 2 \text{ for each } 1\leq i < j \leq m \right\},$$
    %$\left\{ a \in \mathbb{R} _{\ge 0}^m : \sum a _i \le \frac{5}{2} \right\} \cap T$,
    then $f_0 = f_1$.
\end{corollary}
\begin{proof}
   Obviously $\mc C_m$ is symmetric and convex. To show $\mc C_m\subset \ol{\mc N}$,
   it remains for us to consider Condition \eqref{eq.30} when
   $\max\{a_i\}>1$.
   %From $a_i+a_j\leq 2$, immediately it follows that
   %\begin{equation}
   %a_ia_j\leq \left(\f{a_i+a_j}{2}\right)^2\leq 1.
   %\end{equation}
   %i.e. all vectors in $\mc C_m$ satisfy \eqref{eq.29}.
   %On the other hand, when $a_i\leq 1$ for all $i=1,\cdots,m$, Condition \eqref{eq.30}
   %automatically holds.
   Due to the symmetry of $\mc C_m$,
   we can assume $a_1=\max\{a_i\}$ without loss of generality. As shown in the proof of Lemma \ref{le.1},
   this condition is equivalent to
   \begin{equation}
   \f{1}{1-a_1}+G(a_2,\cdots,a_m)\leq m-1.
   \end{equation}
   Here the definition of $G$ is given in \eqref{def-G}.
   For each given $t\in (0,1]$, let
   \begin{equation}
   \mc{D}_{m-1,t}:=\left\{(a_2,\cdots,a_m)\in \R_{\geq 0}^{m-1}\Big|\sum_{i=2}^m a_i\leq 2-t-\f{1}{m-1},\max_i a_i\leq 1-t \right\},
   \end{equation}
   then
   \begin{itemize}
   \item For $a_1=1+t$, $(a_1,\cdots,a_m)\in \mc C_m$ if and only if $(a_2,\cdots,a_m)\in \mc{D}_{m-1,t}$;
   \item $\mc{D}_{m-1,t}$ is a convex polyhedron in $\R^{m-1}$;
   \item $G$ is a symmetric, strictly convex function on $\mc{D}_{m-1,t}$, which should takes its maximum at
   a vertex of this polyhedron.
   \end{itemize}
   Therefore
   \begin{equation}\aligned
   &\f{1}{1-a_1}+G(a_2,\cdots,a_m)\\
   \leq &\sup\left\{\f{1}{1-(1+t)}+\max G|_{\mc{D}_{m-1,t}}:t\in (0,1]\right\}\\
   \leq &\sup\left\{\f{1}{1-(1+t)}+\f{1}{1-(1-t)}+\f{1}{1-\f{m-2}{m-1}}:t\in (0,1]\right\}\\
   = &m-1.
   \endaligned
   \end{equation}
   This completes the proof of $\mc C_m\subset \ol{\mc N}$.
   Finally, $f_0=f_1$ is a direct corollary of Theorem \ref{main-thm}.

\end{proof}
\begin{corollary}\label{coro.2}
 Suppose that $f_0,f_1:\Omega (\subset M)\rightarrow N$ are minimal maps with the same boundary data, $f_0$ is homotopic to $f_1$ and $K_N\le 0$.
    If both $\lambda ^2(df_0)$ and $\lambda ^2(df_1)$ lie in
    $$\mc{V}_m:=\left\{a\in \R_{\geq 0}^m\Big|\prod_{i=1}^m (1 + a _i)^{\frac{1}{2}} \le \mu_m,a_i+a_j\leq 2\text{ for each }1\leq i < j\leq m\right\}$$
    %$$\mc{V}_m:=\left\{ a \in \mathbb{R} _{\ge 0}^m : \prod (1 + a _i)^{\frac{1}{2}} \le 2, \right\},$$
    with
    $$\mu_m:=\sqrt{3}\cdot\left(2-\frac{1}{m-1}\right)^{\frac12},$$
    then $f_0 = f_1$.
\end{corollary}
\begin{proof}
Based on Corollary \ref{coro.1}, it suffices for us to show $\mc{V}_m\subset \mc{C}_m$; equivalently, 
 \begin{equation}\label{eq}
\prod_{i=1}^m (1 + a _i)^{\frac{1}{2}} \le \mu_m \Longrightarrow \sum_{i=1}^m a_i\leq 3-\f{1}{m-1}
\end{equation}
always holds for each $a:=(a_1,\cdots,a_m)$ satisfying $a_1\geq \cdots\geq a_m\geq 0$ and $a_1+a_2\leq 2$.
We shall prove \eqref{eq} by using reduction to absurdity. Assume $\sum\limits_{i=1}^m a_i> 3-\f{1}{m-1}$,
then
\begin{equation}\aligned
&\prod_{i=1}^m (1+a_i)\geq 1+\sum_{i=1}^m a_i+(a_1+a_2)\left(\sum_{i=3}^m a_i\right)\\
>& 1+\left(3-\f{1}{m-1}\right)+2\left(1-\f{1}{m-1}\right)=3\left(2-\f{1}{m-1}\right),
\endaligned
\end{equation}
causing a contradiction. This completes the proof of the present corollary.

\end{proof}

For any vector $a$ satisfying $\prod\limits_{i=1}^{m} (1 + a_i)^{\frac{1}{2}} \le \sqrt{3}$, we have
\begin{equation}
a_i+a_j\leq \sum_{i=1}^m a_i\leq \prod_{i=1}^m(1+a_i)-1\leq 2.
\end{equation}
In conjunction with Corollary \ref{coro.1}-\ref{coro.2}, we can establish a uniqueness result for minimal maps via the slope functions.

\begin{corollary}\label{coro.3}
    Suppose that $f_0,f_1:\Omega (\subset M)\rightarrow N$ are minimal maps with the same boundary data, $f_0$ is homotopic to $f_1$ and $K_N\le 0$.
    If their singular values satisfy $\prod\limits_{i=1}^m (1 + \lambda _i^2)^{\frac{1}{2}} \le \sqrt{3}$, then $f_0 = f_1$.
\end{corollary}

\centerline{}
\centerline{}
\bibliographystyle{amsplain}

\end{document}